\documentclass[a4paper,11pt]{amsart}

\usepackage{amstext,amsfonts,amsthm,graphicx,amssymb,amscd,epsfig}
\usepackage{amsmath}
\usepackage[ansinew]{inputenc}    

\usepackage{mathrsfs}

\theoremstyle{definition}
\newtheorem{definition}{Definition}[section]
\newtheorem{ex}[definition]{Example}
\newtheorem{rem}[definition]{Remark}

\theoremstyle{plain}
\newtheorem{prop}[definition]{Proposition}
\newtheorem{lem}[definition]{Lemma}
\newtheorem{coro}[definition]{Corollary}
\newtheorem{teo}[definition]{Theorem}

\newfont{\bbb}{msbm10 scaled\magstephalf}     
\def\C{\mathbb C}
\def\K{\mathbb K}
\def\R{\mathbb R}

\def\R{\mbox{\bbb R}}

\def\O{\mathcal O}

\usepackage[usenames]{color}

\def\m{\mathfrak{m}}

\title[Bi-Lipschitz triviality]{Bi-Lipschitz triviality of function-germs on singular varieties}

\author{R. Oset Sinha, M. A. S. Ruas}

\date{}

\address{Departament de Matem\`atiques,
Universitat de Val\`encia, Campus de Burjassot, 46100 Burjassot,
Spain}

\email{raul.oset@uv.es}

\address{Instituto de Ci\^encias Matem\'aticas e da Computa\c{c}\~ao,
Universidade de S\~ao Paulo, S\~ao Carlos, SP,
Brazil}

\email{maasruas@icmc.usp.br}

\thanks{Work of R. Oset Sinha partially supported by Grant PID2021-124577NB-I00 funded by MCIN/AEI/ 10.13039/501100011033 and by ``ERDF A way of making Europe"}
\thanks{Work of M. A. S. Ruas partially supported by FAPESP - Grant 2019/21181-0 and by CNPq Grant 305695/2019- 3.}


\begin{document}
\begin{abstract}
In this paper we study the bi-Lipschitz triviality of deformations of an analytic function germ $f$ defined on a germ of an analytic variety $(X, 0)$ in $\mathbb C^n$. We introduce the notion of strongly rational $\mathscr R_X$-bi-Lipschitz trivial families and give an infinitesimal criterion which is a sufficient condition for the bi-Lipschitz triviality of deformations of $f$ on $(X,0).$ As a corollary it follows that when $X$ and $f$ are homogeneous of the same degree, all deformation of $f$ of the same or higher degrees are bi-Lipschitz trivial. We then prove a rigidity result for deformations of $f$ on $X$ when both are weighted homogeneous with respect to the same set of weights.

\end{abstract}

\maketitle

\section{Introduction}

Studying triviality of families of complex function germs $f:(\C^n,0)\to(\C,0)$ for different equivalence relations has been of interest ever since Whitney proved the existence of moduli in the analytic classification of 4 lines in the plane. In the last 20 years many researchers have been attracted by the bi-Lipschitz classification of function germs, which lies somewhere in between the topological and the analytic classification. This is mainly due to the paper by Henry and Parusinski \cite{henryparusinski}, where they give an example of non-bi-Lipschitz-trivial family of function germs and define an invariant of bi-Lipschitz equivalence of function germs in $\C^2$. Since then, giving sufficient conditions for bi-Lipschitz triviality or understanding when the bi-Lipschitz classification is rigid (i.e. when bi-Lipschitz triviality implies analytic triviality) is a cherished goal.

On the other hand, another topic which has gained great impulse in recent years is the study of functions on singular varieties (see \cite{biviakostasruas,limaoreficenunotomazella,oreficenunotomazella}, for example). In this paper we combine both topics. We study the bi-Lipschitz triviality of deformations $F:(\C^n\times \C,0)\to(\C,0)$ of an analytic function germ $f(x)=F(x,0)$ defined on a germ at the origin of an analytic variety $(X, 0)$ in $\mathbb C^n$. For functions with isolated singularity on $X,$ the Bruce-Roberts number $\mu_{BR}(f,X)$ is finite, and
this condition is equivalent to the finite determinacy of $f$ with respect to the group $\mathscr R_X$ of diffeomorphisms
in $(\mathbb C^n, 0),$ preserving $(X,0).$ When $X$ and $f$ are weighted homogeneous with respect to the same set of weights,
and $\mu_{BR}(f,X)$ is finite, it follows from Damon's result in \cite{Drv} that deformations of $f$ by terms of non-negative weights are topologically trivial. In Section \ref{trivial} we address the question of bi-Lipschitz equisingularity of these families for which we introduce the notion of strongly rational $\mathscr R_X$-bi-Lipschitz-trivial families. Based on an extension to bi-Lipschitz equisingularity of an infinitesimal criterion given by the second author and Tomazella in \cite{ruastomazella} for the topological equisingularity of families of function germs defined on $X,$ we prove a sufficient condition for the strongly rational bi-Lipschitz
triviality of deformations of weighted homogeneous germs defined on $(X,0).$ As a corollary, when $X$ and $f$ are homogeneous, it follows that all deformation of $f$ of the same or higher degrees are bi-Lipschitz trivial.

We then address in Section \ref{rigid} the question of rigidity of weighted homogeneous (and not homogeneous) relative deformations of $f,$  extending partially to the relative case the results given by Fernandes and the second author in the non-relative case in \cite{fernandesruas}. More precisely, we prove that if $X$ and $f$ are weighted homogeneous with respect to the same set of weights, with certain hypotheses on the weights, then strongly rational $\mathscr R_X$-bi-Lipschitz equivalence implies $\mathscr R_X$ equivalence. To prove this we need a Thom-Levine type result for $\mathscr R_X$ equivalence, which we prove in Section \ref{sectionthomlevine}. We also show that reduction to the plane curve case is an invariant for the strongly rational $\mathscr R_X$-bi-Lipschitz-triviality of the pair $(F,X).$


\section{Preliminaries}

Denote by $\O_d$ the local ring of germs of smooth functions in $d$ variables over $\C$, and denote its maximal ideal by $\m_d$. Let $X\subset \C^n$ be a complex variety and let $\Theta_X$ be the $\O_n$-module of tangent vector fields to $X$. We say that two analytic function-germs $f,g:(\C^n,0)\to (\C,0)$ are $\mathscr R_{X}$-analytically equivalent if there exists a bi-analytic germ $\phi:(\C^n,0)\to (\C,0)$ which preserves $X$ such that $g=f\circ\phi$. Such bi-analytic germs are obtained by integrating analytic vector fields in $\Theta_X$.

Similarly, we say that two analytic function-germs $f,g:(\C^n,0)\to (\C,0)$ are $\mathscr R_{X}$-bi-Lipschitz equivalent if there exists a bi-Lipschitz germ $\phi:(\C^n,0)\to (\C,0)$ which preserves $X$ such that $g=f\circ\phi$.

We will need the following classical results

\begin{lem}\label{mather}{\rm \bf (Mather's Lemma)}
Let $\alpha:G\times M\to M$ be a smooth  action of a Lie group $G$ on a manifold $M$, and let $V$ be a connected
submanifold of $M$. Then $V$ is contained in a single orbit if and only if the following hold:

{\rm (a)} $T_vV\subseteq T_v(G.v), \forall v\in V$,

{\rm (b)} $\dim T_v(G_v)$ is independent of $v\in V$.
\end{lem}

\begin{teo}\label{thomlevine}{\rm \bf (Thom-Levine)}
Let $U$ be a domain in $\C$, $W$ a neighbourhood of 0 in $\C^n$ and $F:W\times U\to\C$ such that $F(0,t)=0$. Write $f_t(x)=F(x,t)$, for all $t\in U$ and for all $x\in W$. If there is a family of analytic vector fields $v:W\times U\to \C^n$ with $v(0,t)=0$ for all $t\in U$ and $$\frac{\partial f_t}{\partial t}(x)=(df_t)(x)(v(x,t))$$ for all $t\in U$ and for all $x\in W$, then $f_t$ is analytically equivalent to $f_{t'}$ for any $t,t'\in U$.
\end{teo}

We refer to \cite{nunomond} for more details on all these definitions and results.

On one hand, we will prove a Thom-Levine type result for $\mathscr R_X$-analytical triviality in Section \ref{sectionthomlevine}. On the other hand, motivated by the Thom-Levine theorem, we will define strongly rational $\mathscr R_{X}$-bi-Lipschitz triviality in Section \ref{rigid}, which is a relative version of the strong bi-Lipschitz triviality defined by Fernandes and the second author in \cite{fernandesruas}.

\section{Thom-Levine type result for $\mathscr R_X$-analytical triviality}\label{sectionthomlevine}

Let $X\subset \C^n$ be a complex hypersurface singularity given by $\varphi:(\C^n,0)\to (\C,0)$ such that $X=\varphi^{-1}(0)$ and let $\Theta_X=\langle \eta_1,\ldots,\eta_r\rangle$ be the $\O_n$-module of tangent vector fields to $X$. Denote by $\Theta_X^0$ the submodule of $\Theta_X$ given by all the vector fields which vanish at the origin. Let $f:(\C^n,0)\to (\C,0)$ be a function-germ with isolated singularity at 0. Define the $\mathscr R_{X,e}$-codimension (also known as the Bruce-Roberts Milnor number) as $$\mathscr R_{X,e}-cod(f)=\mu_{BR}(f,X)=\dim_{\C}\frac{\O_n}{df(\Theta_X)}.$$ Similarly, define the $\mathscr R_{X}$-codimension as $$\mathscr R_{X}-cod(f)=\dim_{\C}\frac{\m_n}{df(\Theta_X^0)}.$$ Let $s=rk\{\eta_1(0),\ldots,\eta_r(0)\}$. The number $s$ is also known as the dimension of the analytic stratum, i.e. the dimension of the tangent space to the iso-singular locus. Notice that if ${0}$ is a stratum of $X$ then $s=0$ and in this case $\Theta_X=\Theta_X^0$. Notice also that $s\leq n$ and $s=n$ if and only if $X=\C^n$.

\begin{lem}\label{cods}
Let $f:(\C^n,0)\to (\C,0)$ be a function-germ with isolated singularity at 0. The following hold
\begin{enumerate}
\item[i)] $$\dim_{\C}\frac{df(\Theta_X)}{df(\Theta_X^0)}=s,$$
\item[ii)] $\mathscr R_{X}-cod(f)=\mathscr R_{X,e}-cod(f)+s-1=\mu_{BR}(f,X)+s-1.$
\end{enumerate}
\end{lem}
\begin{proof}
Let $\xi_1,\ldots,\xi_s$ be $s$ linearly independent generators in $\Theta_X$ such that $\xi_i(0)\neq 0$ for $i=1,\ldots,s$. Choose $r-s$ linearly independent vector fields $\eta_j'$, $j=s+1,\ldots,r$, in $\Theta_X$ such that $\Theta_X=\langle\xi_1,\ldots,\xi_s,\eta_{s+1}',\ldots,\eta_r'\rangle$. Now $$\frac{df(\Theta_X)}{df(\Theta_X^0)}=\frac{\langle\xi_1(f),\ldots,\xi_s(f),\eta_{s+1}'(f),\ldots,\eta_r'(f)\rangle}{\m_n\langle\xi_1(f),\ldots,\xi_s(f)\rangle + \langle\eta_{s+1}'(f),\ldots,\eta_r'(f)\rangle}\approx\frac{\langle\xi_1(f),\ldots,\xi_s(f)\rangle}{\m_n\langle\xi_1(f),\ldots,\xi_s(f)\rangle}.$$ We write $\xi_j=\sum_{i=1}^n\xi_{j_i}\frac{\partial}{\partial x_i}$, so the above quotient is equal to $$Sp_{\C}\{\overline{\sum_{i=1}^n\xi_{1_i}\frac{\partial f}{\partial x_i}},\ldots,\overline{\sum_{i=1}^n\xi_{s_i}\frac{\partial f}{\partial x_i}}\},$$ where $\overline{A}$ means the class of $A$ in the quotient.

We follow now the steps of the proof of Lemma 3.1 in \cite{nunomond}. Suppose that these classes are linearly dependent. Then there exist $c_1,\ldots,c_s\in\C$ not all zero such that $$\sum_{j=1}^sc_j\overline{\sum_{i=1}^n\xi_{j_i}\frac{\partial f}{\partial x_i}}=0,$$ and therefore, there exist $\alpha_1,\ldots,\alpha_s\in\m_n$ such that $$\sum_{j=1}^sc_j\sum_{i=1}^n\xi_{j_i}\frac{\partial f}{\partial x_i}=\sum_{j=1}^s\alpha_j\sum_{i=1}^n\xi_{j_i}\frac{\partial f}{\partial x_i}.$$ This implies that $$0=\sum_{j=1}^s(c_j-\alpha_j)\sum_{i=1}^n\xi_{j_i}\frac{\partial f}{\partial x_i}=\sum_{i=1}^n(\sum_{j=1}^s(c_j-\alpha_j)\xi_{j_i})\frac{\partial f}{\partial x_i}.$$

Consider the vector field $$v=\sum_{i=1}^n(\sum_{j=1}^s(c_j-\alpha_j)\xi_{j_i})\frac{\partial}{\partial x_i}.$$ Suppose that all the components of $v$ vanish at 0. Since $\alpha_j(0)=0$, $j=1,\ldots,s$, we get
\[
\begin{matrix}
c_1\xi_{1_1}(0)+\ldots+c_s\xi_{s_1}(0)=0\\
\vdots\\
c_1\xi_{1_n}(0)+\ldots+c_s\xi_{s_n}(0)=0\\
\end{matrix}
\]
We have that $s\leq n$ and $rk(\xi_{i_j})_{1\leq i\leq s,1\leq j\leq n}=s$, so $c_1=\ldots=c_s=0$, which is a contradiction. Therefore, not all the components of $v$ vanish at the origin, and $v$ is non-singular. Since $df(v)=0$, f has critical points over an integral curve at 0 and thus, f has a non-isolated singularity, which is again a contradiction. Hence, $\overline{\xi_1(f)},\ldots,\overline{\xi_s(f)}$ are linearly independent and $$\dim_{\C}\frac{df(\Theta_X)}{df(\Theta_X^0)}=s.$$

For the second part, consider the short exact sequence $$0\rightarrow \frac{df(\Theta_X)}{df(\Theta_X^0)}\rightarrow \frac{\m_n}{df(\Theta_X^0)}\rightarrow \frac{\m_n}{df(\Theta_X)}\rightarrow 0.$$ Using i) we get $\mathscr R_{X}-cod(f)=\mathscr R_{X,e}-cod(f)-1+\dim_{\C}\frac{df(\Theta_X)}{df(\Theta_X^0)}=\mathscr R_{X,e}-cod(f)+s-1=\mu_{BR}(f,X)+s-1.$
\end{proof}

\begin{rem}

\begin{enumerate}
\item[i)] If $\{0\}$ is a stratum of $X$, $\mathscr R_{X}-cod(f)=\mathscr R_{X,e}-cod(f)-1$.
\item[ii)] If $X=\C^n$ we recover the non relative version of the relation between the Milnor number and the codimension of the $\mathscr R$-orbit of $f$, $\mathscr R-cod(f)=\mathscr R_{e}-cod(f)+n-1=\mu(f)+n-1$, which can be found in Lemma 3.1 in \cite{nunomond}, for example.
\end{enumerate}
\end{rem}

\begin{definition}
Suppose that $\eta_1,\ldots,\eta_r$ generate $\Theta_X$ in a neighbourhood $U$ of $0\in\mathbb C^n$ and let $T_U^*\mathbb C^n$ be the restriction of the cotangent bundle on $\mathbb C^n$ to $U$. Consider $$LC_U(X)=\{(x,\delta)\in T_U^*\mathbb C^n:\delta(\eta_i(x))=0,i=1,\ldots,r\}.$$ The \emph{logarithmic characteristic variety} $LC(X,0)$ is the germ at the origin of $LC_U(X)$ in $T^*\mathbb C^n$ at $T_0^*\mathbb C^n$.
\end{definition}

\begin{definition}
Suppose that $\eta_1,\ldots,\eta_r$ generate $\Theta_X$. Given $f:(\C^n,0)\to (\C,0)$ a function-germ with isolated singularity at 0 and $F:(\C^n\times \C,0)\to (\C,0)$ a deformation of $f$ such that $F(0,t)=0$. We say that $F$ is a \emph{good deformation} of $f$ if $$C(F)=\{(x,t)\in\C^n\times \C:dF(\eta_i)=0,i=1,\ldots,r\}$$ does not split, i.e. $C(F)=\{0\}\times \C$ near $(0,0)$.
\end{definition}

\begin{definition}
Given $(w_{1},...,w_{n}:d_{1},...,d_{p})$, $w_{i}$, $d_{j}\in
{\mathbb Q}^{+}$, a map germ $f:\K^{n},0\rightarrow \K^{p},0$ is
\emph{weighted homogeneous} of type $(w:d)=(w_{1},...,w_{n}:d_{1},...,d_{p})$
if for all $\lambda \in
\K-\{0\}$:\[f(\lambda^{w_{1}}x_{1},\lambda^{w_{2}}x_{2},...,
\lambda^{w_{n}}x_{n})=
(\lambda^{d_{1}}f_{1}(x),\lambda^{d_{2}}f_{2}(x),...,
\lambda^{d_{p}}f_{p}(x)).\]

For function germs we will sometimes forget $d$ and consider weighted homogeneous function germs of weight $w=(w_{1},...,w_{n})$.
\end{definition}

We can now prove a Thom-Levine type result for $\mathscr R_X$-analytical triviality of families.
\begin{prop}\label{thomlevine}
Let $X$ be a hypersurface such that $LC(X,0)$ is Cohen-Macaulay. Let $U$ be a domain in $\C$ and $F(x,t)$ a family of polynomials such that for every $t\in U$, $f_t(x)=F(x,t)$ is weighted homogeneous with weight $w$ and has an isolated singularity at 0. If $F$ is a good deformation of $f_0$ and $\frac{\partial F}{\partial t}\in T\mathscr R_X(f_t)=df_t(\Theta_X^0)$ for each fixed $t\in U$, then $f_{t_1}$ is $\mathscr R_X$-analytically equivalent to $f_{t_2}$ for all $t_1,t_2\in U$.
\end{prop}
\begin{proof}
Consider $H_w^d(n,1)$ the space of weighted homogeneous polynomials in $n$-variables and of weights $w$. For a large enough $m$, $H_w^d(n,1)$ can be seen as a subset of the space of $m$-jets $M=J^m(n,1)$. The set $A_w^d(n,1)$ of weighted homogeneous polynomials with isolated singularity at 0 is a Zariski open subset of $H_w^d(n,1)$. Consider the action on $H_w^d(n,1)$ given by $G=j^m(\mathscr R_X)$. For $f\in A_w^d(n,1)$, by Lemma \ref{cods}, we have
\begin{align*}
codim(G\cdot f)&=\dim(M)-\dim(G\cdot f)=\dim(M)-\dim(T_f(G\cdot f))  \\
&=\dim\frac{\m_n}{\m_{n}^{m+1}}-\dim\frac{T\mathscr R_X(f)}{\m_{n}^{m+1}}=\dim\frac{\m_n}{T\mathscr R_X(f)}\\
&=\mathscr R_X-cod(f)=\mu_{BR}(f,X)+s-1.
\end{align*}
Now, by Theorem 3.7 in \cite{ahmedruastomazella}, if $LC(X,0)$ is Cohen-Macaulay and $F$ is a good deformation, then $\mu_{BR}(f_t,X)$ does not depend on $t$ (it is constant in the family), therefore, $\dim T_{f_t}(G\cdot f_t)$ does not depend on $t$.

Set $P=\{f_t:t\in U\}$. By hypothesis $T_{f}P\subset T_f(G\cdot f)$ for all $f\in P$, and by the above discussion $\dim T_{f}(G\cdot f)$ is constant for all $f\in P\subset A_w^d(n,1)$, therefore, by Mather's Lemma, $P$ is contained in a single $G$-orbit, i.e. $f_{t_1}$ is $\mathscr R_X$-analytically equivalent to $f_{t_2}$ for all $t_1,t_2\in U$.
\end{proof}

\begin{rem}
The above result holds in the particular case that the hypersurface $X$ is a free divisor, since by \cite{bruceroberts} $LC(X)$ is Cohen-Macaulay. Also, we can ensure the constancy of the Bruce-Roberts Milnor number in the family (and thus remove the hypotheses of $LC(X)$ being Cohen-Macaulay and $F$ being a good deformation) if $X$ is a hypersurface with isolated singularity by \cite{oreficenunotomazella}, and more generally if $X$ is an ICIS by \cite{limaoreficenunotomazella}.
\end{rem}

\section{Criterion for $\mathscr R_X$-bi-Lipschitz triviality}\label{trivial}

In Theorem 3.4 in \cite{ruastomazella} the second author and Tomazella gave a sufficient condition for a family to be topologically $\mathscr R_X$-trivial. Their proofs adapt well to the bi-Lipschitz setting. We state the results in this section but we only include a sketch of the proof for the main result and do not include the proofs of the other results since they are direct adaptations of the proofs in \cite{ruastomazella}. For this we need the following definitions. In this section $\mathbb K=\mathbb R$ or $\mathbb C$.

\begin{definition}
A \emph{control function} $\rho:\mathbb K^n\to \R$ is a real analytic, irreducible function such that $\rho(x)\geq 0$ and $\rho(x)=0$ if and only if $x=0$.
\end{definition}

In what follows we consider control functions of the form $\rho=\sum_{i=1}^{n}x_i^{a_i}\overline{x_i}^{a_i}=\sum_{i=1}^{n}\vert x_i\vert^{2a_i}$, unless stated otherwise.

\begin{definition} Let $f:(\mathbb K^n,0)\to (\mathbb K,0)$ be a function-germ with isolated singularity at 0 and $f_t=F:(\mathbb K^n\times \mathbb K,0)\to (\mathbb K,0)$ a deformation of $f$. $F$ is said to be \emph{strongly rational $\mathscr R_X$-bi-Lipschitz trivial} if there exists a control function $\rho$ and a family of real analytic vector fields $v_t=\sum_{i=1}^nv_{i_t}\frac{\partial}{\partial x_i}$ tangent to $X$ such that $V_t=\sum_{i=1}^n\frac{v_{i_t}}{\rho}\frac{\partial}{\partial x_i}$ satisfies that $\frac{v_{i_t}}{\rho}$ are Lipschitz functions where the Lipschitz constant does not depend on $t$ and $$\frac{\partial F}{\partial t}=dF(V_t).$$ This implies that $$\rho\frac{\partial F}{\partial t}=\sum_{j=1}^r\alpha_jdF(\eta_j),$$ for some real analytic functions $\alpha_j$ and where the $\eta_j$ are the generators of $\Theta_X$.

Notice that while $\frac{v_{i_t}}{\rho}$ are Lipschitz functions, $\frac{\alpha_j}{\rho}$ need not be Lipschitz.
\end{definition}

\begin{ex}\label{ex1}

\begin{itemize}
\item[i)] Let $X\subset\mathbb K^3$ be given by $\varphi(u,v,w)=2u^5-v^7+w^5$. 
Consider $F=u^5+v^7+2w^5+tu^3v^4$, a deformation of $f=u^5+v^7+2w^5$. The family $F$ is not ${\mathscr R}$-trivial since $u^3v^4$ is not in $Jf$, and hence not ${\mathscr R}_{X}$-trivial. Consider the control function $\rho=\vert u\vert^{20}+|v|^{28}+|w|^{20}$. One can check using Singular that $u^{10}\frac{\partial F}{\partial t}=dF(\xi_1),v^{14}\frac{\partial F}{\partial t}=dF(\xi_2),w^{10}\frac{\partial F}{\partial t}=dF(\xi_3)$ for some $\xi_1,\xi_2,\xi_3\in \Theta_X$ and so $$\rho\frac{\partial F}{\partial t}=dF(\overline{u}^{10}\xi_1+\overline{v}^{14}\xi_2+\overline{w}^{10}\xi_3).$$ We write $$dF(\xi_i)=\beta_1^idF(\eta_e)+\beta_2^idF(\eta_2)+\beta_3^idF(\eta_3)+\beta_4^idF(\eta_4)$$ for $i=1,2,3$ and use
Singular to obtain the $\beta_j^i$. Now,
\begin{align*}
&dF(\overline{u}^{10}\xi_1+\overline{v}^{14}\xi_2+\overline{w}^{10}\xi_3)=\\
&(\overline{u}^{10}\beta_1^1+\overline{v}^{14}\beta_1^2+\overline{w}^{10}\beta_1^3)dF(\eta_e)+\ldots+(\overline{u}^{10}\beta_4^1+\overline{v}^{14}\beta_4^2+\overline{w}^{10}\beta_4^3)dF(\eta_e)
\end{align*}   
and so we get $\rho\frac{\partial F}{\partial t}=\sum_{j=1}^r\alpha_jdF(\eta_j)$ for some real analytic (not polynomial) $\alpha_j$. From the $\alpha_j$ we get a real analytic vector field $v_t=\sum_{i=1}^nv_{i_t}\frac{\partial}{\partial x_i}$ tangent to $X$ such that $V_t=\sum_{i=1}^n\frac{v_{i_t}}{\rho}\frac{\partial}{\partial x_i}$ satisfies $\frac{\partial F}{\partial t}=dF(V_t).$ It then just remains to check that $\frac{v_{i_t}}{\rho}$ are Lipschitz functions for all $i$. Therefore, $F=u^5+v^7+2w^5+tu^3v^4$ is strongly rational $\mathscr R_X$-bi-Lipschitz trivial.

\item[ii)]
Let $X\subset\mathbb K^4$ be given by $\varphi(x,y,z,w)=x^2+y^2+z^2+w^3$ with filtration $((3,3,3,2),6)$. $\Theta_X$ is generated by
\begin{align*}
\eta_{e}=&3x\frac{\partial}{\partial x}+3y\frac{\partial}{\partial y}+3z\frac{\partial}{\partial z}+2w\frac{\partial}{\partial w},\\
\eta_{2}=&2y\frac{\partial}{\partial x}-2x\frac{\partial}{\partial y},\eta_{3}=2z\frac{\partial}{\partial x}-2x\frac{\partial}{\partial z},\\
\eta_{4}=&3w^2\frac{\partial}{\partial x}-2x\frac{\partial}{\partial w},\eta_{5}=3w^2\frac{\partial}{\partial y}-2y\frac{\partial}{\partial w},\\
\eta_{6}=&2z\frac{\partial}{\partial y}-2y\frac{\partial}{\partial z},\eta_{7}=3w^2\frac{\partial}{\partial z}-2z\frac{\partial}{\partial w},
\end{align*}
Consider $f=2x^2-y^2-3z^2+w^3$, and a deformation $F=2x^2-y^2-3z^2+w^3+tx^2w$. The family $F$ is not ${\mathscr R_X}$-trivial since $x^2w$ is not in $dF(\Theta_X)$. Taking $\rho=|x|^{24}+|y|^{24}+|z|^{24}+|w|^{36}$ and with similar calculations to item i), one can see that $F=2x^2-y^2-3z^2+w^3+tx^2w$ is strongly rational $\mathscr R_X$-bi-Lipschitz trivial.
\end{itemize}
\end{ex}

\begin{rem}
\begin{itemize}
\item[i)] Two function-germs $f,g:(\mathbb K^n,0)\to (\mathbb K,0)$ are $\mathscr R_X$-bi-Lipschitz equivalent if there exists a bi-Lipschitz map germ $\phi:(\mathbb K^n,0)\to (\mathbb K^n,0)$ which preserves $X$ and such that $f=g\circ\phi$. If $F$ is a strongly rational $\mathscr R_X$-bi-Lipschitz trivial, then for small $t\neq t'$, $f_t$ is $\mathscr R_X$-bi-Lipschitz equivalent to $f_{t'}$. This follows from the proof of the Thom-Levine theorem since the flow of the Lipschitz vector field tangent to $X$ defines a family of homeomorphisms which preserve $X$.

\item[ii)] In \cite{camararuas} C\^amara and the second author prove that in the non relative case, for weighted homogeneous (not homogeneous) plane curves, bi-Lipschitz triviality implies strongly bi-Lipschitz triviality.
\end{itemize}

\end{rem}

In what follows we can assume that $df_{t}\xi(0)=0, \forall \, \xi
\in \Theta_{X}$ and $F(0,t)=0$ for all $t$. In fact, if $\xi \in
\Theta_{X}$, then $df_{t}\xi.\frac{\partial F}{\partial
t}=df_{t}(\frac{\partial F}{\partial t}.\xi)$. If
$df_{t}\xi_{0}(0)\neq 0$ for some $\xi_{0}$, then
\[
\frac{\partial F}{\partial t}=df_{t}(\frac{\frac{\partial
F}{\partial t}.\xi_{0}}{ df_{t}\xi_{0}})\] and hence the
deformation is analytically ${\mathscr R}_{X}$-trivial. Observe that $\frac{\frac{\partial
F}{\partial t}.\xi_{0}}{ df_{t}\xi_{0}}\in \Theta_{X}^{0}$.

\begin{definition}
\begin{enumerate}
\item[a)] We assign weights $w_1,..,w_n$, $w_{i} \in
{\mathbb Q}^{+}$, $i=1,...,n$  to a given coordinate system
$x_1,...,x_n$ in $\K^n$. The filtration of a monomial
$x^{\beta}=x_{1}^{\beta_{1}}x_{2}^{\beta_{2}}...x_{n}^{\beta_{n}}$
with respect to this set of weights is defined by
 fil$(x^{\beta})=\sum_{i=1}^{n}\beta_{i}w_{i}$.

\item[b)] We define a filtration in the ring ${\mathcal O}_{n}$ via the
function
 \[{\rm fil}(f)={\rm inf}_{|\beta|}\{\mbox{fil}(x^{\beta}) :
\frac{\partial^{|\beta|}f}{\partial x^{\beta}}(0)\neq 0\},\,
|\beta|=\beta_1+\ldots+\beta_n.\] The filtration of a map germ
$f=(f_1,\ldots,f_p):\K^n,0 \to \K^p,0$ is fil$(f)=(d_1,\ldots,d_p)$,
where fil$(f_i)=d_i$.

\item[c)] We extend the filtration to $\Theta_X$, defining
$w(\frac{\partial }{\partial x_j})=-w_j$ for all $j=1,..,n$, so
that given $\xi=\sum^{n}_{j=1}\xi_{j}\frac{\partial }{\partial
x_j} \in \Theta_X$, then fil$(\xi)={\rm
inf}_j\{\mbox{fil}(\xi_j)-w_j\}$.

\end{enumerate}
\end{definition}

 Let $w=w_1w_2\cdots w_n$, ${\bf
w}=(w_1,...,w_n)$, and $\|x\|_{\bf
w}=(|x_1|^{2w/w_1}+...+|x_n|^{2w/w_n})^{1/2w}$.

In what follows  $A \lesssim B$ means there is some positive
constant $C$ with $A\leq C B$.

\begin{teo} \label{centralpesado} Let ${\bf w}=(w_1,..,w_n),\, w_1 \geq w_2 \ldots \geq w_n$ be an
$n$-tuple of positive integers which are the weights of $x_1,\ldots,x_n$. Let $\eta_1,\ldots, \eta_m$ be
a system of generators for $\Theta_{X}^0$ and $d_i={\rm
fil}(\eta_i)$, $i=1,\ldots,m$. Let $f_0:\K^n,0 \to \K,0$ be a
${\mathscr R}_{X}$-finitely determined germ and $F:\K^n\times \K,0
\to \K,0$  a good deformation of $f_0$. If
\[| \frac{\partial F}{\partial t}|\lesssim  \sup_{i=1,...,m}
\{|df_t(\eta_i)| \|x\|_{\bf w}^{-d_{i}+w_1-w_n}\}\,\, {\rm for} \,x(\neq
0)\,{\rm near}\, 0\]
 then $F$ is strongly rational ${\mathscr R}_{X}$-bi-Lipschitz trivial.
\end{teo}
\begin{proof}  We choose non negative integers $e_i$, $i=1,\ldots ,m$ so
that $d_i+e_i$ is a constant $s$. We define a function $\rho$ by
$\rho^2=\sum_{i=1}^{m}|\rho_i|^2\|x\|_{\bf w}^{2e_i}$, where
$\rho_i=df_t(\eta_i)$, $i=1,\ldots, m$. Since $F$ is a good
deformation it follows that $ C(\rho(x,t))=\{0\}\times \K$. We have
$\rho^2 \frac{\partial F}{\partial t}=df_{t}\left(\frac{\partial
F}{\partial t}\sum^{m}_{i=1}\overline{\rho_{i}}\|x\|_{\bf
w}^{2e_i}\eta_i \right)$, so if we consider $V$ to be the vector field in $\K^n\times
\K,0$ tangent to $X$ defined by
\[V(x,t)=\left\{
\begin{array}{ll}
 -\frac{1}{\rho^2}\frac{\partial F}{\partial
t}(\overline{\rho_{1}}
 \|x\|_{\bf w}^{2e_1}\eta_{1}+...+\overline{\rho_{m}}
 \|x\|_{\bf w}^{2e_m} \eta_{m}) &
 \mbox{if $x\neq 0$}\\
0 & \mbox{if $x=0$}
\end{array}
\right. \] we have $\frac{\partial F}{\partial t}=df_{t}
\left( V \right)$. The vector field $V(x,t)$ is real analytic away from
$\{0\}\times \K$. For $j=1,..,m$ and $i=1,..,m$, let $V_j$ denote
the $j$-th component of $V$, and let $\eta_{ij}$ denote the
$j$-th component of $\eta_i$. Then
\[V_j(x,t)=-\frac{1}{\rho^2}\frac{\partial F}{\partial t}
\left( \overline{\rho_{1}} \|x\|_{\bf w}^{2e_1}\eta_{1j}+...+
\overline{\rho_{m}}\|x\|_{\bf w}^{2e_m} \eta_{mj} \right).\]
 Since fil$(\eta_i)=d_i$, we have
fil$(\eta_{ij})\geq d_i+w_j$, thus $|\eta_{ij}|\lesssim
\|x\|_{\bf w}^{d_i+w_j}$. Then,
\[\begin{array}{ll}|V_j(x,t)|
  &\lesssim
\frac{1}{\rho} |\frac{\partial F}{\partial t}|\|x\|_{\bf w}^{e_1}
\|x\|_{\bf w}^{d_1+w_j} +...+\frac{1}{\rho} |\frac{\partial
F}{\partial t}|  \|x\|_{\bf w}^{e_m}\|x\|_{\bf w}^{d_m+w_j}\\\\
 &\lesssim\frac{1}{\rho} |\frac{\partial
F}{\partial t}|\|x\|_{\bf w}^{s} \|x\|_{\bf w}^{w_j}\\\\
 & \lesssim
\frac{1}{\rho}\sup_i\{|\rho_i|\|x\|_{\bf w}^{-d_{i}}\}\|x\|_{\bf
w}^{s} \|x\|_{\bf w}^{w_j+w_1-w_n} \lesssim  \|x\|_{\bf w}^{w_j+w_1-w_n}.
\end{array}\]
 It follows that $|V_j(x,t)|\leq C \|x\|_{\bf w}^{w_j+w_1-w_n}$, for
$j=1,\ldots,n$ and this implies that the vector field $V$ is continuous.

To show $V_j(x,t)$ is a Lipschitz function for all $j=1, \ldots, n,$ it is sufficient
to verify directly that  for each $k=1, \ldots, n,$
$$\text{fil} ({\frac{\partial V_j}{\partial x_k}}(x,t)) \lesssim  \|x\|_{\bf w}^{w_j+w_1-w_n-w_k} \lesssim  \|x\|_{\bf w}^{w_j-w_n}.$$
Since $w_j\geq w_n$ for all $j=1, \ldots, n,$ this implies that $\frac{\partial V_j}{\partial x_k}(x,t)$ are bounded at 0. The result follows.

\end{proof}


\begin{definition}\label{gerahomo} A germ of an analytic variety $X,0 \subseteq
\K^{n},0$ is weighted homogeneous if it is defined by a weighted
homogeneous map germ $f:\K^{n},0\rightarrow \K^{p},0$. A set of
generators  $\{\eta_{1},...,\eta_{m}\}$ of $\Theta_{X}$ is
weighted homogeneous  of type $(w_{1},...,w_{n}: d_{1},...,d_{m})$
if $\eta_{i}=\sum_{j=1}^{n}\eta_{ij} \frac{\partial}{\partial
x_{j}}$, $w(x_{j})=w_{j}$, $w(\frac{\partial}{\partial
x_{j}})=-w_{j}$ and
fil$(\eta_{i})=d_{i}=\mbox{fil}(\eta_{ij})-w_{j}$, does not
depend on $j=1,...,n$, for each $i=1,\ldots,m$.
\end{definition}

When $X$ is a weighted homogeneous variety, we can always  choose
weighted homogeneous generators for $\Theta_{X}$ (see
\cite{damonfreeness}). Given a weighted homogeneous vector field $\eta=\sum_{j=1}^n\eta^j\frac{\partial}{\partial x_j}$, since the filtration of all the components is the same we call this filtration $\deg(\eta)$.

\begin{ex}
The Euler vector field $\eta_e=\sum_{j=1}^nw_jx_j\frac{\partial}{\partial x_j}$ has $\deg(\eta_e)=0$.
\end{ex}

Given $f$ a polynomial such that $\mbox{fil}(f)=d$, then $\deg(\eta(f))=\deg(df(\eta))=\deg(\eta)+\deg(f)=\deg(\eta)+d$.

The following result is an adaptation of a result by J. Damon in
\cite{Drv}. We include it here as a corollary of Theorem
\ref{centralpesado}.

\begin{coro}\label{homotri} Let $X$ be a weighted homogeneous subvariety
of ${\K}^n,0$ with weights $(w_1,\ldots,w_n)$, $w_1\geq\ldots\geq w_n$, and let $f_0:{\K}^n,0  \to {\K},0$ be weighted
homogeneous with the same weights as $X$, of degree $d$ and ${\mathscr R}_{X}$-finitely
determined. Then any deformation $F$ of $f_0$ by terms of
filtration greater than or equal to $d+w_1-w_n$ is strongly rational
${\mathscr R}_{X}$-bi-Lipschitz trivial.
\end{coro}

\begin{rem}\label{homo}
In particular, if $f_0$ is a homogeneous ${\mathscr R}_{X}$-finitely
determined germ of degree $d$ and $F:\K^{n}\times \K,0\rightarrow \K,0$ is a
deformation of $f_{0}$ by terms of filtration greater than or equal to $d$, then it is strongly rational ${\mathscr R}_{X}$-bi-Lipschitz trivial (see Corollary 3.10 in \cite{ruastomazella}).
\end{rem}

\begin{ex}\label{ex48}

\begin{itemize}
\item[i)] Take $F$ and $X$ as in Example \ref{ex1} i) with filtration $((7,5,7),35)$. 
The filtration of $u^3v^4$ is $41>37=35+7-5=d+w_1-w_n$ and by Corollary \ref{homotri}, the family is strongly rational ${\mathscr R}_{X}$-bi-Lipschitz trivial.
\item[ii)] Take $F$ and $X$ as in Example \ref{ex1} i) with filtration $((3,3,3,2),6)$. The filtration of $x^2w$ is $8>7=6+3-2=d+w_1-w_n$ and by Corollary \ref{homotri}, the family is strongly rational ${\mathscr R}_{X}$-bi-Lipschitz trivial.
\end{itemize}
\end{ex}

\section{$\mathscr R_X$-bi-Lipschitz rigidity}\label{rigid}

In this section we show that, under a certain condition, in the weighted homogeneous case strongly rational $\mathscr R_X$-bi-Lipschitz triviality implies $\mathscr R_X$-triviality.




Let $X=\varphi^{-1}(0)\subset \C^n$ be a complex hypersurface singularity given by a weighted homogenous $\varphi:(\C^n,0)\to (\C,0)$ of type $((w_1,\ldots,w_n);D)$ such that $LC(X,0)$ is Cohen-Macaulay and let $\Theta_X=\langle \eta_1,\ldots,\eta_r\rangle$ where the $\eta_i$ are all weighted homogeneous. We set $\eta_1=\eta_e$ to be the Euler vector field. Let $F(x,t)$ be a family of polynomial functions such that $f_t:(\C^n,0)\to (\C,0)$ is a weighted homogeneous function-germ with isolated singularity at 0 of type $((w_1,\ldots,w_n);d)$ for all $t\in U$ a domain in $\C$. By Corollary 3.10 in \cite{ruastomazella} the deformation $F$ is a good deformation.

We will deal with the case where $\Theta_X=\Theta_X^0$ and when there is a generic polar curve such that the equations that define it are analytically reduced. Without loss of generality we will assume the polar curve is given by $\Gamma=\{(x,t):\varphi=0,dF(\eta_{2})=\ldots=dF(\eta_{n-1})=0\}$. We also assume that there is no branch of the polar curve in $\{\varphi=0, x_n=0\}$. This assumption is related to the definition of \emph{miniregular} given in Section 4 of \cite{henryparusinski}.

\begin{teo}\label{nvar}
In the above setting, suppose that $w_1\geq w_2\geq \ldots\geq w_{n-1}>w_n$. If $F$ is strongly rational $\mathscr R_X$-bi-Lipschitz trivial, then $F$ is $\mathscr R_X$-analytically trivial. Moreover, if $\deg(\eta_e)<\deg(\eta_{i})$, $i=2,\ldots,r$ and $F(x,t)=f(x)+t\theta(x)$, then $\theta=cf$, where $c$ is a constant.
\end{teo}
\begin{proof}
If $F$ is strongly rational $\mathscr R_X$-bi-Lipschitz trivial, there exists a control function $\rho$ of degree $2k$ big enough and there exists a Lipschitz vector field $V$ tangent to $X$ such that $\frac{\partial F}{\partial t}=dF(V)$, where $\rho V=\alpha_1\eta_e+\ldots+\alpha_r\eta_r$. This means that $$\rho\frac{\partial F}{\partial t}=\alpha_1dF(\eta_e)+\ldots+\alpha_rdF(\eta_r).$$

There exist non negative integers $\epsilon_i$, $i=2,\ldots,r$, such that $\deg(\eta_i)=\epsilon_i$ and so $\deg(dF(\eta_i))=d+\epsilon_i$. Since the left hand side of the above equation is a polynomial of degree $2k+d$, necessarily $\deg(\alpha_i)=2k-\epsilon_i$ or $\alpha_i=0$.

We consider a generic polar curve on $X$ which is not included in $F^{-1}(0)$ given by $\Gamma=\{(x,t):\varphi=0,dF(\eta_2)=\ldots=dF(\eta_{n-1})=0\}$.

Since there is no branch of the polar curve $\Gamma_t$ in $\{\varphi=0,x_n=0\}$, we call $(a_{1_j}(t),\ldots,a_{{(n-1)}_j}(t))$ the roots of $\{\varphi(x_1,\ldots,x_{n-1},1,t)=0,dF(\eta_2)(x_1,\ldots,x_{n-1},1,t)=\ldots=dF(\eta_{n-1})(x_1,\ldots,x_{n-1},1,t)=0\}$. The branches of the polar curve can be parameterised as $$\gamma_j(s)=(a_{1_j}(t)s^{w_1},\ldots,a_{{(n-1)}_j}(t)s^{w_{n-1}},s^{w_n},t).$$


Since $V_n=\frac{v_n}{\rho}$ is Lipschitz
\begin{align*}
&|\frac{v_n}{\rho}(\gamma_j(s))-\frac{v_n}{\rho}(\gamma_l(s))|\leq \lambda(t)|\gamma_j(s)-\gamma_l(s)|\\
&\leq \lambda(t)(|a_{1_j}(t)-a_{1_l}(t)||s|^{w_1}+\ldots+|a_{{(n-1)}_j}(t)-a_{{(n-1)}_l}(t)||s|^{w_{n-1}})\\
&\leq \lambda(t)(|a_{1_j}(t)-a_{1_l}(t)||s|^{w_1-w_{n-1}}+\ldots+|a_{{(n-1)}_j}(t)-a_{{(n-1)}_l}(t)|)|s|^{w_{n-1}}\\
&\leq \tilde\lambda(t)|s|^{w_{n-1}}.
\end{align*}

On the other hand, on the polar curve, $\sum_{i=1}^{n-1}v_i\frac{\partial F}{\partial x_i}=\alpha_1dF(\eta_e)+\alpha_ndF(\eta_n)+\ldots +\alpha_rdF(\eta_r)$, and so $v_i=\alpha_1\eta_e^i+\alpha_n\eta_n^i+\ldots+\alpha_r\eta_r^i$. Therefore, we have
\begin{align*}
\frac{v_n}{\rho}&=\frac{\frac{\partial F}{\partial t}-\sum_{i=1}^{n-1}\frac{v_i}{\rho}\frac{\partial F}{\partial x_i}}{\frac{\partial F}{\partial x_n}}=\frac{\rho\frac{\partial F}{\partial t}-\sum_{i=1}^{n-1}v_i\frac{\partial F}{\partial x_i}}{\rho\frac{\partial F}{\partial x_n}}\\
&=\frac{\rho\frac{\partial F}{\partial t}-\sum_{i=1}^{n-1}(\alpha_1\eta_e^i+\alpha_n\eta_n^i+\ldots+\alpha_r\eta_r^i)\frac{\partial F}{\partial x_i}}{\rho\frac{\partial F}{\partial x_n}}.
\end{align*}


Now, the numerator of the previous equation is a polynomial of degree $2k+d$ and the denominator is a polynomial of degree $2k+d-w_n$, therefore
$$|\frac{v_n}{\rho}(\gamma_j(s))-\frac{v_n}{\rho}(\gamma_l(s))|=|k_j(t)-k_l(t)||s|^{w_n},$$ where $k_j$ and $k_l$ are constant in $s$. We then have that $|k_j(t)-k_l(t)||s|^{w_n}\leq \tilde\lambda(t)|s|^{w_{n-1}}$, and this is only possible if $k_j(t)=k_l(t)$, i.e. this constant term in $s$ does not depend on the branch of the polar curve. So, on $\Gamma$, we can write $\frac{v_n}{\rho}(x,t)=k(t)L(x_n,\overline{x}_n)$, where $L$ is a linear function (since $w_n$ is the lowest weight and $v_n$ is real analytic).

On the other hand $k(t)L(x_n,\overline{x}_n)=\frac{v_n}{\rho}=\frac{\alpha_1}{\rho}\eta_e^n+\frac{\alpha_n}{\rho}\eta_n^n+\ldots+\frac{\alpha_r}{\rho}\eta_r^n$.

STEP 1) Suppose first that $\deg(\eta_e)<\deg(\eta_i)$, $i=n,\ldots,r$. Then the only vector field out of $\eta_e,\eta_n,\ldots,\eta_r$ with a multiple of $x_n$ (and hence weight $w_n$) in the last component is the Euler vector field, so we get that on the polar curve $\frac{\alpha_1}{\rho}w_nx_n=k(t)L(x_n,\overline{x}_n)$, so $L(x_n,\overline{x}_n)$ must be $x_n$ and so $\frac{\alpha_1}{\rho}=\frac{k(t)}{w_n}$.

This implies that on each branch of the polar curve we have $$\frac{\partial F}{\partial t}-\frac{k(t)}{w_n}dF(\eta_e)=\frac{\alpha_n}{\rho}dF(\eta_n)+\ldots+\frac{\alpha_r}{\rho}dF(\eta_r).$$ The term $k(t)$ is an algebraic expression of the functions $a_{j_i}(t)$, however, following Remark 3.3 in \cite{fernandesruas}, with an appropriate change of variables which depends on the Puiseux expansions of the $a_{j_i}(t)$ we can make it analytic. For simplicity we keep the notation and suppose $k(t)$ is complex analytic.

The left hand side of the above expression is a complex polynomial of degree $d$, so the right hand side must be a complex polynomial. The degree of $\alpha_j$ is $2k-\epsilon_j$, and since $\deg(\eta_e)<\deg(\eta_i)$, $i=n,\ldots,r$, $\epsilon_j>0$. Therefore $\rho$ does not divide $\alpha_j$, $j=n,\ldots,r$. On the other hand the degree of $dF(\eta_j)$ is $d+\epsilon_j<2k$, so $\rho$ does not divide $dF(\eta_j)$ either. Since $\rho$ is irreducible, this means that $\rho$ does not divide $\alpha_jdF(\eta_j)$ for $j=n,\ldots,r$ and thus it does not divide $\alpha_ndF(\eta_n)+\ldots+\alpha_rdF(\eta_r)$. Hence, this polynomial must be 0 and so, on the polar curve $$\frac{\partial F}{\partial t}-\frac{k(t)}{w_n}dF(\eta_e)=0.$$

Now, since the equations of the polar curve are analytically reduced, we get that on $X$  $$\frac{\partial F}{\partial t}-\frac{k(t)}{w_n}dF(\eta_e)=\lambda_2(x,t)dF(\eta_2)+\ldots\lambda_{n-1}(x,t)dF(\eta_{n-1}),$$ where the $\lambda_i$ are complex polynomials in $\O_{n+1}$. For those $\eta_i$ such that $\deg(\eta_e)<\deg(\eta_i)$, the degrees of the $dF(\eta_i)$ are greater than the degree of the polynomial on the left hand side. Call $\eta_{i_1},\ldots,\eta_{i_m}$ those $\eta_i$, $i\in\{2,\ldots,n-1\}$, such that $\deg(\eta_e)=\deg(\eta_i)$. Then $\sum_{j\neq i_1,\ldots,i_m}\lambda_{j}(x,t)dF(\eta_{j})=0$.

Finally we get that, on $X$, $$\frac{\partial F}{\partial t}=\frac{k(t)}{w_n}dF(\eta_e)+\sum_{j=1}^{m}\lambda_{i_j}(x,t)dF(\eta_{i_j}).$$ This means that $\frac{\partial F}{\partial t}\in dF(\Theta_X^0)$ and by Proposition \ref{thomlevine} $F$ is analytically trivial.

Furthermore, if $\deg(\eta_e)<\deg(\eta_{i})$, $i=2,\ldots,r$ and $F(x,t)=f(x)+t\theta(x)$, we get $\frac{\partial F}{\partial t}=\theta(x)=\frac{k(t)}{w_n}dF(\eta_e)=\frac{k(t)d}{w_n}\cdot F$, and since $\theta$ does not depend on $t$, necessarily $\theta(x)=c\cdot f(x)$, for a constant $c$.

STEP 2) Suppose that there exist $\eta_{j_1},\ldots,\eta_{j_k}$, with $j_1,\ldots,j_k>n-1$, with the same degree as $\eta_e$. Then, we get that on the polar curve $k(t)L(x_n,\overline{x}_n)=(\alpha_1w_nx_n+\alpha_{j_1}\eta_{j_1}^n+\ldots+\alpha_{j_k}\eta_{j_k}^n)/\rho$. Since the $\eta_{j_i}^n$ are complex analytic, necessarily $\eta_{j_i}^n=w_{j_i}x_n$, which forces $L(x_n,\overline{x}_n)=ax_n$ for a constant $a$. Thus $\alpha_1/\rho=ak(t)/w_n-(\alpha_{j_1}w_{j_1}/w_n+\ldots+\alpha_{j_k}w_{j_k}/w_n)/\rho$. We then have
$$\frac{\partial F}{\partial t}=\frac{ak(t)}{w_n}dF(\eta_e)+\frac{\alpha_{j_1}}{\rho}(\frac{w_{j_1}}{w_n}dF(\eta_{e})+dF(\eta_{j_1}))+\ldots+\frac{\alpha_{j_k}}{\rho}(\frac{w_{j_k}}{w_n}dF(\eta_{e})+dF(\eta_{j_k})).$$

    Since the left hand side is a polynomial, so is the right hand side, and so $\frac{\alpha_{j_i}}{\rho}=k_{j_i}(t)$ are constant in $s$ for $i=1,\ldots,k$. Again, considering the Puiseux expansion of the $k_{j_i}(t)$ we can use a common change of variable to make all of them (together with $k(t)$) analytic.
    In this case we proceed as in the proof of STEP 1 and we get that on the polar curve $$\frac{\partial F}{\partial t}-(\frac{ak(t)}{w_n}+k_{j_1}(t)\frac{w_{j_1}}{w_n}+\ldots+k_{j_k}(t)\frac{w_{j_k}}{w_n})dF(\eta_e)-k_{j_1}(t)dF(\eta_{j_1})-\ldots-k_{j_k}(t)dF(\eta_{j_k})=0.$$ Now continue as in STEP 1.

\end{proof}

\begin{rem}
\begin{enumerate}
\item[i)] The hypothesis of the equations that define the polar curve $\Gamma=\{(x,t):\varphi=0,dF(\eta_2)=\ldots=dF(\eta_{n-1})=0\}$ being analytically reduced is equivalent to asking the ideal given by $\{\varphi,dF(\eta_2),\ldots,dF(\eta_{n-1})\}$ to be radical.
\item[ii)] In the case $X=\C^{n-1}$ we have $\Theta_X=\langle\frac{\partial}{\partial x_1},\ldots,\frac{\partial}{\partial x_{n-1}},x_n\frac{\partial}{\partial x_{n}}\rangle$. The hypothesis $w_1>\ldots>w_n$ used in \cite{fernandesruas} implies $\deg(\frac{\partial}{\partial x_1})=-w_1<\ldots<\deg(\frac{\partial}{\partial x_{n-1}})=-w_{n-1}<\deg(x_n\frac{\partial}{\partial x_n})=0$. In particular, when $X=\C^2$ then we recover the result in \cite{fernandesruas}.
\item[iii)] In the case $n=2$ we consider the whole of $X$ to be the polar curve.
\end{enumerate}
\end{rem}

\begin{ex}
Consider $X_{\text {hom}}\subset\C^3$ given by $\varphi_{\text {hom}}(x,y,z)=x^2+y^2+z^2$ and take the family $F_{\text {hom}}(x,y,z)=2x^2-y^2-3z^2+tx^2$. By Remark \ref{homo}, $F_{\text hom}$ is strongly rational $\mathscr R_{X_{\text {hom}}}$-bi-Lipschitz trivial for some tangent vector field $v_{t_{\text {hom}}}$.

One may ask what happens when we add a new variable and break the homogeneity but keep a weighted homogeneity. Consider now $X\subset\C^4$ given by $\varphi(x,y,z,w)=x^2+y^2+z^2+w^3$ as in Example \ref{ex1} and $F(x,y,z,w)=2x^2-y^2-3z^2+w^3+tx^2$. Certainly the equation $\frac{\partial F}{\partial t}=dF(v_t/\rho)$ can be solved for some control function $\rho$ and a real analytic vector field $v_t$ tangent to $X$ such that when $w=0$ we recover the $v_{t_{\text {hom}}}$ of the homogeneous case above. The problem is that the new component $v_{t_4}$ may not satisfy that $\frac{v_{t_4}}{\rho}$ is a Lipschitz function. A simple study of the degrees of the components of $v_t$ show that the partial derivatives of $\frac{v_{t_4}}{\rho}$ with respect to $x,y,z$ have to be zero in order to be bounded.

We recall that $\Theta_X$ is generated by
\begin{align*}
\eta_{e}=&3x\frac{\partial}{\partial x}+3y\frac{\partial}{\partial y}+3z\frac{\partial}{\partial z}+2w\frac{\partial}{\partial w},\\
\eta_{2}=&2y\frac{\partial}{\partial x}-2x\frac{\partial}{\partial y},\eta_{3}=2z\frac{\partial}{\partial x}-2x\frac{\partial}{\partial z},\\
\eta_{4}=&3w^2\frac{\partial}{\partial x}-2x\frac{\partial}{\partial w},\eta_{5}=3w^2\frac{\partial}{\partial y}-2y\frac{\partial}{\partial w},\\
\eta_{6}=&2z\frac{\partial}{\partial y}-2y\frac{\partial}{\partial z},\eta_{7}=3w^2\frac{\partial}{\partial z}-2z\frac{\partial}{\partial w},
\end{align*}

Firstly notice that $\deg(\eta_2)=\deg(\eta_3)=\deg(\eta_6)=\deg(\eta_e)$. A generic polar curve is given by $\Gamma=\{(x,y,z,w,t):\varphi=dF(\eta_2)=dF(\eta_7)=0\}=\{(x,y,z,w,t): \varphi=xy=zw^2=0\}$, so the hypotheses of Theorem \ref{nvar} are satisfied. However, we cannot use this Theorem directly to know whether $F$ is strongly rational $\mathscr R_X$-bi-Lipschitz trivial.

Even so we may consider $X_0\subset\C^2$ given by $\varphi_0(x,w)=x^2+w^3$, which is obtained from $X$ by taking the section $y=z=0$. This is a branch of the polar curve $\Gamma$. Notice that projecting vector fields in $\Theta_X$ to the $x,w$-plane gives exactly $\Theta_{X_0}=\langle 3x\frac{\partial}{\partial x}+2w\frac{\partial}{\partial w},\eta_{4}=3w^2\frac{\partial}{\partial x}-2x\frac{\partial}{\partial w}\rangle$, therefore if $F(x,y,z,w)=2x^2-y^2-3z^2+w^3+tx^2$ were strongly rational $\mathscr R_X$-bi-Lipschitz trivial, then $F_0(x,w)=2x^2+w^3+tx^2$ would be strongly rational $\mathscr R_{X_0}$-bi-Lipschitz trivial. However, $X_0$ and $F_0$ satisfy the hypotheses of Theorem \ref{nvar} since $\deg(\eta_4)>\deg(\eta_e)$, but $\theta=x^2\neq c(2x^2+w^3)$ for some constant $c$. Therefore $F_0$ is not strongly rational $\mathscr R_{X_0}$-bi-Lipschitz trivial and so $F$ is not strongly rational $\mathscr R_{X}$-bi-Lipschitz trivial.

On one hand this example shows that reduction to a homogeneous case does not ensure strongly rational $\mathscr R_{X}$-bi-Lipschitz triviality. On the other hand, it suggests how cases where our hypotheses are not satisfied can be reduced to cases where they are.
\end{ex}

This example suggests the following

\begin{definition}
Let $X=\varphi^{-1}(0)\subset \C^n$ be a complex hypersurface singularity given by a weighted homogenous $\varphi:(\C^n,0)\to (\C,0)$ of type $((w_1,\ldots,w_n);D)$ and let $F(x,t)$ be a family of polynomial functions such that $f_t:(\C^n,0)\to (\C,0)$ is a weighted homogenous function-germ with isolated singularity at 0 of type $((w_1,\ldots,w_n);d)$. Let $H=\{x\in\C^n : x_{i_1}=\ldots=x_{i_k}=0\}$ be the transversal intersection of coordinate hyperplanes for $1<k<n$. We say that $H$ is \emph{$(X,F)$-generic} if $X_0=X\cap H$  and $F_0=F|_{H}$ satisfy that $F_0$ is not homogeneous and that $\mu_{BR}(F_0,X_0)<\infty$. Notice that in this setting $\pi(\Theta_X)=\Theta_{X_0}$, where $\pi:\theta_n\to\theta_{n-k}$ is the projection of vector fields in $\theta_n$ to $H$ (making $x_{i_1}=\ldots=x_{i_k}=0$ and eliminating the corresponding components).
\end{definition}

This definition seems natural in the bi-Lipschitz context. Compare it for example with the definition of admissible $(n-k+1)$-planes in \cite{denkowskitibar}.

\begin{prop}
Let $X=\varphi^{-1}(0)\subset \C^n$ be a complex hypersurface singularity given by a weighted homogenous $\varphi:(\C^n,0)\to (\C,0)$ and $F(x,t)$ a family of polynomial functions such that $f_t:(\C^n,0)\to (\C,0)$ is weighted homogenous of the same type of $X$ with isolated singularity at 0. Suppose that $H$ is an $(X,F)$-generic $(n-k)$-plane. If $F$ is strongly rational $\mathscr R_{X}$-bi-Lipschitz trivial, then $F_0$ is strongly rational $\mathscr R_{X_0}$-bi-Lipschitz trivial.
\end{prop}
\begin{proof}
Suppose that $F$ is strongly rational $\mathscr R_{X}$-bi-Lipschitz trivial, then there exists a control function $\rho$ and a family of real analytic vector fields $v_t$ tangent to $X$ such that $V_t=\sum_{i=1}^n\frac{v_{i_t}}{\rho}\frac{\partial}{\partial x_i}$ satisfies that $\frac{v_{i_t}}{\rho}$ are Lipschitz functions and $$\frac{\partial F}{\partial t}=dF(V_t).$$ Since $\pi(\Theta_X)=\Theta_{X_0}$, $v_{t_0}=\pi(v_t)$ is tangent to $X_0$. For simplicity suppose $H=\{x\in\C^n : x_{1}=\ldots=x_{k}=0\}$ and set $\rho_0(x_{k+1},\ldots,x_n)=\rho(0,\ldots,0,x_{k+1},\ldots,x_n)$, then $\frac{v_{i_{t_0}}}{\rho_0}$ are Lipschitz functions and $$\frac{\partial F_0}{\partial t}=dF_0(V_{t_0}),$$ where $V_{t_0}=\sum_{i=k+1}^n\frac{v_{i_{t_0}}}{\rho_0}\frac{\partial}{\partial x_i}$. Hence $F_0$ is strongly rational $\mathscr R_{X_0}$-bi-Lipschitz trivial.
\end{proof}

\begin{coro}
Let $X$, $F$ and $H=\{x_1=\ldots=x_{n-2}=0\}$ be as above, suppose that $X_0$ is a plane curve with isolated singularity with $w_{n-1}>w_{n}$ and $\deg(\eta_2)>\deg(\eta_e)$. If $F_0=f_0+t\theta_0$ satisfies that $\theta_0\neq cf_0$ for some constant $c$, then $F$ is not strongly rational $\mathscr R_{X}$-bi-Lipschitz trivial.
\end{coro}
\begin{proof}
Notice that in this case $n-k=2$, so if $w_{n-1}>w_{n}$ the hypotheses of Theorem \ref{nvar} are satisfied. Therefore, if $\theta_0\neq cf_0$ for some constant $c$, then $F_0$ is not strongly rational $\mathscr R_{X_0}$-bi-Lipschitz trivial and by the above Proposition $F$ is not strongly rational $\mathscr R_{X}$-bi-Lipschitz trivial.
\end{proof}

The case in which the equations of the polar curve are not analytically reduced has to be treated separately and we can only prove the same result when $n=3$.

\begin{teo}\label{comps}
Suppose that $w_1\geq w_2>w_3$ and that $\deg(\eta_e)<\deg(\eta_i)$, $i=3,\ldots,r$. If $F:(\mathbb C^3\times \mathbb C,0) \to (\mathbb C,0) $ is strongly rational $\mathscr R_X$-bi-Lipschitz trivial, then $F$ is $\mathscr R_X$-analytically trivial. Moreover, if $\deg(\eta_e)<\deg(\eta_i)$, $i=2,\ldots,r$ and $F(x,t)=f(x)+t\theta(x)$, then $\theta=cf$, where $c$ is a constant.
\end{teo}
\begin{proof}
The proof is analogous to the proof of Theorem \ref{nvar} until we used the fact of the equations of the polar curve being analytically reduced. The only difference is that now the generic polar is given by $\Gamma=\{\varphi=0,dF(\eta_2)=0\}$. As above we obtain that $\frac{\partial F}{\partial t}-\frac{k(t)}{w_n}dF(\eta_e)=0$ on the polar curve. From here on we follow the steps of the proof of Theorem 3.4 in \cite{fernandesruas} adapted to the relative case. Let $P_1(x,t),\ldots,P_z(x,t)$ be analytic irreducible factors of $dF(\eta_2)$ in $\O_4$. Let $\alpha_i=\max\{\alpha:P_i^{\alpha} \text{ divides } dF(\eta_2) \text{ in } \O_4\}$. There exist positive integers $\beta_1,\ldots,\beta_z$ and a function $u(x,t)\in\O_4$ such that, on $X$, $$\frac{\partial F}{\partial t}=\frac{k(t)}{w_n}dF(\eta_e)+uP_1^{\beta_1}\ldots P_z^{\beta_z},$$ where $P_i$ does not divide $u$. Our objective is to show that $\alpha_i\leq \beta_i.$

Set $$u_i=uP_1^{\beta_1}\ldots \widehat{P_i^{\beta_i}}\ldots P_z^{\beta_z}.$$ We can now write $$\frac{\partial F}{\partial t}=\frac{k(t)}{w_n}dF(\eta_e)+u_iP_i^{\beta_i},$$ where $P_i$ does not divide $u_i$. Suppose $\alpha_i>1$ (otherwise $\beta_i\geq \alpha_i$ trivially). We can now apply $\eta_2$ to the previous equation. First notice that since $\eta_2$ does not involve the variable $t$, $\frac{\partial}{\partial t}$ and $\eta_2$ commute. On the other hand $dF(\eta_e)=d\cdot F$. Using this and the fact that $\eta_2$ is a derivation we obtain $$\frac{\partial }{\partial t}(dF(\eta_2))=\frac{k(t)}{w_n}d\cdot dF(\eta_2)+\eta_2(u_i)P_i^{\beta_i}+\beta_iu_iP_i^{\beta_i-1}\eta_2(P_i).$$ Now, $P_i^{\alpha_i}$ divides $dF(\eta_2)$, so $P_i^{\alpha_i-1}$ divides $\frac{\partial }{\partial t}(dF(\eta_2))$ and $\frac{k(t)d}{w_n}\cdot dF(\eta_2)$. Hence $P_i^{\alpha_i-1}$ divides $\eta_2(u_i)P_i^{\beta_i}+\beta_iu_iP_i^{\beta_i-1}\eta_2(P_i)=P_i^{\beta_i-1}(\eta_2(u_i)P_i+\beta_iu_i\eta_2(P_i)).$ Since $P_i$ does not divide $u_i$ or $\eta_2(P_i)$, we get $\beta_i\geq \alpha_i$.

Finally, we can now write $$\frac{\partial F}{\partial t}=\frac{k(t)}{w_n}dF(\eta_e)+q(x,t)dF(\eta_2),$$ where $q$ is an analytic function. Looking at the degrees of the different factors we conclude that $q=q(t)$ or $q=0$ depending on whether $\deg(\eta_2)=\deg(\eta_e)$ or not and we follow as in STEP 1 of the proof of Theorem \ref{nvar}.
\end{proof}

\begin{rem}\label{rem3}
In the case there exists $\eta_{i_0}\neq\eta_2$ such that $\deg(\eta_e)=\deg(\eta_{i_0})<\deg(\eta_j)$, $j\neq i_0$, as seen in STEP 2 of the proof of Theorem \ref{nvar}, we get
$$\frac{\partial F}{\partial t}-\frac{k(t)}{w_n}dF(\eta_e)=\frac{\alpha_{i_0}}{\rho}\frac{w'}{w_n}dF(\eta_{e})+\frac{\alpha_{i_0}}{\rho}dF(\eta_{i_0}).$$
Now, $\eta_2$ and $\eta_{i_0}$ may not commute, so in order for the proof of Theorem \ref{comps} to apply we would need $P_i^{\alpha_i-1}$ to divide $k'(t)\eta_2(dF(\eta_{i_0}))=k'(t)(\eta_{i_0}(dF(\eta_2))-[\eta_2,\eta_{i_0}](F))$. Since $\eta_{i_0}$ is a derivation, $P_i^{\alpha_i-1}$ divides the first factor, so we would need $P_i^{\alpha_i-1}$ to divide the Lie bracket $[\eta_2,\eta_{i_0}](F)$.
\end{rem}

\begin{coro}
Let $X=\varphi^{-1}(0)$ be an isolated hypersurface weighted homogeneous singularity in $\mathbb C^3$ of type $((w_1,w_2,w_3),D)$ with $w_1\geq w_2>w_3$ and suppose that $mult(\varphi)\geq 3$. Then we have rigidity of strong rational $\mathscr R_X$-bi-Lipschitz triviality.
\end{coro}
\begin{proof}
If $X$ has isolated singularity, then we have that $\Theta_X=\langle\eta_e,\eta_{12},\eta_{13},\eta_{32}\rangle$, where $\eta_e$ is the Euler vector field of degree 0 and $\eta_{ij}=\frac{\partial \varphi}{\partial x_j}\frac{\partial}{\partial x_i}-\frac{\partial \varphi}{\partial x_i}\frac{\partial}{\partial x_j}$ are the Hamiltonian vector fields of degree $D-(w_i+w_j)$ (see \cite{bruceroberts}, for example). Therefore $\deg(\eta_e)\leq \deg(\eta_{12})<\deg(\eta_{13})\leq \deg(\eta_{23}).$

Suppose $\deg(\eta_{12})=\deg(\eta_e)=0$, then $D=w_1+w_2$ and so $\varphi=x_1x_2+\tilde\varphi$. Therefore $mult(\varphi)\geq 3$ implies that $\deg(\eta_{12})>\deg(\eta_e)$ and we can apply Theorem \ref{comps}.
\end{proof}

\begin{ex}
Consider the cross-cap $X\subset\mathbb C^3$ given by $\varphi(u,v,w)=v^2-u^2w$. Consider the filtration $((1,2,2),4)$. $\Theta_X$ is generated by
\begin{align*}
\eta_{e}=&u\frac{\partial}{\partial u}+2v\frac{\partial}{\partial v}+2w\frac{\partial}{\partial w},\\
\eta_{2}=&u\frac{\partial}{\partial u}+v\frac{\partial}{\partial v},\\
\eta_{3}=&u^2\frac{\partial}{\partial v}+2v\frac{\partial}{\partial w},\\
\eta_{4}=&v\frac{\partial}{\partial u}+uw\frac{\partial}{\partial v}.
\end{align*}
Here $\deg(\eta_{e})=\deg(\eta_2)=\deg(\eta_3)=0<1=\deg(\eta_4)$, so we can use Theorem \ref{nvar}. Consider the family $F=u^6+v^3+w^3+t\theta$ with filtration $((1,2,2),6)$ and take $\Gamma=\{\varphi=dF(\eta_2)=0\}$ as a generic polar curve. We have that $dF(\eta_2)=6u^6+3v^3+tu\frac{\partial\theta}{\partial u}+tv\frac{\partial\theta}{\partial v}$, which is reduced, so $\alpha_1$ in Theorem \ref{comps} is 1. Therefore, we have that $P_1^{\alpha_1-1}=1$ and it divides $[\eta_2,\eta_{3}](F)$, and following Remark \ref{rem3} we have rigidity. In other words, any $\theta$ of weighted degree 6 for which $F$ is strongly rational $\mathscr R_X$-bi-Lipschitz trivial will be a combination of $dF(\eta_e),dF(\eta_2)$ and $dF(\eta_3)$ and will therefore be $\mathscr R_X$-analytically trivial too.

Consider now the same example but with the filtration $((2,3,2),6)$. We must consider $\eta_e'=2u\frac{\partial}{\partial u}+3v\frac{\partial}{\partial v}+2w\frac{\partial}{\partial w}$, instead of $\eta_e$. Here the lowest weight is not unique and so we cannot apply our results. Consider the family $F=u^3+v^2+w^3+tu^2w$ with filtration $((2,3,2),6)$. It is reasonable to expect this to be an example of strongly rational $\mathscr R_X$-bi-Lipschitz trivial which is not $\mathscr R_X$-analytically trivial. However, suppose it is strongly rational $\mathscr R_X$-bi-Lipschitz trivial with a control function $\rho=|u|^{6r}+|v|^{4r}+|w|^{6r}$ of degree $12r$ and a real analytic vector field $\vec{v}=(v_1,v_2,v_3)$ tangent to $X$ such that $\rho u^2w=v_1\frac{\partial f}{\partial u}+v_2\frac{\partial f}{\partial v}+v_3\frac{\partial f}{\partial w}$ (taking $t=0$ in $\rho\theta=dF(\vec{v}_t)$). This equation has degree $12r+6$ and hence $\deg(v_1)=\deg(v_3)=12r+2$ and $\deg(v_2)=12r+3$. If we want $\frac{\vec{v}}{\rho}$ to be Lipschitz, then the derivative of $\frac{v_1}{\rho}$ with respect to $v$ must be bounded (in particular). Differentiating the quotient we get $$\frac{\frac{\partial v_1}{\partial v}\rho-v_1\frac{\partial \rho}{\partial v}}{\rho^2}.$$ The numerator is a polynomial of degree $24r-1$ and the denominator has degree $24r$, so necessarily this quotient must be 0. Solving the differential equation given by $\frac{\partial v_1}{\partial v}\rho-v_1\frac{\partial \rho}{\partial v}=0$ we get $v_1(u,\overline{u},v,\overline{v},w,\overline{w})=k_1(u,\overline{u},w,\overline{w})\rho(u,\overline{u},v,\overline{v},w,\overline{w})$ or 0. Comparing degrees, $\deg(k_1)=2$, so it must be a linear function in $u,\overline{u},w,\overline{w}$. Similarly $v_3(u,\overline{u},v,\overline{v},w,\overline{w})=k_3(u,\overline{u},w,\overline{w})\rho(u,\overline{u},v,\overline{v},w,\overline{w})$, where $k_3$ is a linear function in $u,\overline{u},w,\overline{w}$ or 0.

Suppose $k_1$ has a factor in $w$, then $v_1=k_1\rho$ would contain the expression $u^{3r}\overline{u}^{3r}w+v^{2r}\overline{v}^{2r}w+w^{3r+1}\overline{w}^{3r}$, but there is no vector in $\Theta_X$ with a pure power of $w$ in the first component. We get the same contradiction supposing $k_1$ has a factor in $\overline{w}$.

Suppose $k_1$ has a factor in $u$, then $v_1=k_1\rho$ would contain the expression $u^{3r+1}\overline{u}^{3r}+uv^{2r}\overline{v}^{2r}+uw^{3r}\overline{w}^3r$. However, \begin{align*}
u^{3r+2}\overline{u}^{3r}w+u^2v^{2r}\overline{v}^{2r}w+u^2w^{3r+1}\overline{w}^{3r}=\rho\theta=v_1\frac{\partial f}{\partial u}+v_2\frac{\partial f}{\partial v}+v_3\frac{\partial f}{\partial w}=\\(u^{3r+1}\overline{u}^{3r}+uv^{2r}\overline{v}^{2r}+uw^{3r}\overline{w}^3r+\tilde{k_1})3u^2+v_22v+v_33w^2,
\end{align*}
and there is no way to cancel the term $3u^{3r+3}\overline{u}^{3r}$ from this equation. We get the same contradiction if we suppose that $k_1$ has a factor in $\overline{u}$. This means that $v_1=0$. Similarly $v_3=0$. Now the equation $\rho\theta=u^{3r+2}\overline{u}^{3r}w+u^2v^{2r}\overline{v}^{2r}w+u^2w^{3r+1}\overline{w}^{3r}=v_1\frac{\partial f}{\partial u}+v_2\frac{\partial f}{\partial v}+v_3\frac{\partial f}{\partial w}=v_22v$ has no solution outside $X$. Even restricting ourselves to $X$, and changing $v^2$ for $u^2w$, we get $\vec{v}=(0,1/2\rho v,0)$, which is not tangent to $X$ for any $r$.

In conclusion, with such a control function it is not possible to have strongly rational $\mathscr R_X$-bi-Lipschitz triviality. We have tried using control functions with different filtrations but we have not succeeded in finding a weighted homogeneous example of strongly rational $\mathscr R_X$-bi-Lipschitz trivial family which is not $\mathscr R_X$-analytically trivial.
\end{ex}

\end{document}